\documentclass[12pt,leqno]{amsart}
\usepackage{amsmath,amsfonts,amssymb,amsthm}
\usepackage{epsfig}

\theoremstyle{plain}
\newtheorem{theorem}{Theorem}
\newtheorem{proposition}[theorem]{Proposition}

\newtheorem{lemma}[theorem]{Lemma}

\theoremstyle{definition}
\newtheorem{definition}[theorem]{Definition}

\newcommand{\ve}{\varepsilon}

\newcommand{\boldM}{{\mathbf M}}

\newcommand{\cL}{{\mathcal L}}

\newcommand{\Hn}{\mathbb H^n}

\newcommand{\field}[1]{\Bbb{#1}}
\newcommand{\G}{\field{G}}                  
\newcommand{\R}{\field{R}}                  
\newcommand{\g}{\mathfrak g}

\def\Barint_#1{\mathchoice
          {\mathop{\vrule width 6pt height 3 pt depth -2.5pt
                  \kern -8pt \intop}\nolimits_{#1}}%
          {\mathop{\vrule width 5pt height 3 pt depth -2.6pt
                  \kern -6pt \intop}\nolimits_{#1}}%
          {\mathop{\vrule width 5pt height 3 pt depth -2.6pt
                  \kern -6pt \intop}\nolimits_{#1}}%
          {\mathop{\vrule width 5pt height 3 pt depth -2.6pt
                  \kern -6pt \intop}\nolimits_{#1}}}

\begin{document}

\title[Riesz potentials and $p$-superharmonic functions, etc.]
{Riesz potentials and $p$-superharmonic functions in Lie groups of Heisenberg type}

\author{Nicola Garofalo}
\address{Department of Mathematics\\Purdue University \\
West Lafayette, IN 47907} \email[N. Garofalo]{garofalo@math.purdue.edu}
\thanks{First author supported in part by NSF Grant DMS-07010001}

\author{Jeremy T. Tyson}
\address{Department of Mathematics \\
University of Illinois at Urbana-Cham\-paign\\
Urbana, IL 61801} \email[J. T. Tyson]{tyson@math.uiuc.edu}
\thanks{Second author supported in part by NSF Grant DMS-0901620}

\maketitle

\begin{abstract}
We prove a superposition principle for Riesz potentials of nonnegative
continuous functions on Lie groups of Heis\-enberg type. More precisely, we
show that the Riesz potential
$$
R_\alpha(\rho)(g) = \int_{\G} N(g^{-1} g')^{\alpha-Q} \rho(g') dg', \qquad 0<\alpha<Q,
$$
of a nonnegative function $\rho\in C_0(\G)$ on a group $\G$ of
Heisenberg type is necessarily either $p$-subharmonic or $p$-superhar\-monic,
depending on $p$ and $\alpha$. Here $N$ denotes the non-isotropic
homogeneous norm on such groups, as introduced by Kaplan.
This result extends to a wide class of nonabelian stratified Lie groups a
recent remarkable superposition result of Lindqvist and Manfredi.
\end{abstract}

\section{Introduction}\label{S:intro}

The study of Riesz potentials  occupies a central position in classical
analysis and potential theory, see  \cite{S1} and \cite{L}. A basic result
states that if one considers the Newtonian potential of $\rho\in C_0(\R^n)$,
$n\ge 3$,
\[
I_2(\rho)(x) = c_n \int_{\R^n} \frac{\rho(y)}{|x-y|^{n-2}} dy,
\]
then $-\Delta(I_2(\rho)) =  \rho$, in other words the operator $I_2$ is the
inverse of minus the Laplacian. In particular, when $\rho \ge 0$, then
$I_2(\rho)$ is a nonnegative superharmonic function in $\R^n$. Vice versa, the
F. Riesz decomposition theorem states that every nonnegative superharmonic
function in $\R^n$ arises---modulo harmonic functions---as the Newtonian
potential of a measure, see \cite{He}. More in general, the Riesz potential
\[
I_\alpha(\rho)(x) = c_{n,\alpha} \int_{\R^n}  \frac{\rho(y)}{|x-y|^{n-\alpha}}
dy,\ \ \ 0<\alpha<n,
\]
provides an inverse for the fractional power of the Laplacian. One has in fact
the following well-known identities in $\mathcal S'(\R^n)$, see \cite{S1}:
\[
(-\Delta)^{\alpha/2}(I_\alpha(\rho)) =  I_\alpha((-\Delta)^{\alpha/2}\rho) = \rho.
\]

What is remarkable is that the Riesz potentials, which as we have seen are
intrinsically connected to a linear operator, the Laplacian, also interact
with a highly nonlinear operator, namely, the $p$-Laplacian
\[
\Delta_p f = \text{div}(|\nabla f|^{p-2} \nabla f),\ \ \ \ \ \ \  1<p <\infty.
\]
This was discovered by Lindquist and Manfredi in \cite{LM}, where the authors
proved the following surprising result.

\begin{theorem}\label{T:LM}
Let $\rho\in C_0(\R^n)$, $\rho \ge 0$. The following three cases hold:
\begin{itemize}
\item[1)] If $2<p<n$, then $I_{n-\alpha}(\rho)$ is $p$-superharmonic when 
\[
0<\alpha \le \frac{n-p}{p-1}.
\]
\item[2)] If $p>n$, then $I_{n-\alpha}(\rho)$ is $p$-subharmonic when 
\[
-\alpha \ge \frac{p-n}{p-1}.
\]
If $p=\infty$, one may take $-\alpha \ge 1.$
\item[3)] If $p = n$, then the function
\[
I_n(x) = \int_{\R^n} \rho(y) \log |x-y| dy
\]
is $n$-subharmonic.
\end{itemize}
\end{theorem}

It is remarkable here that the threshold $\frac{n-p}{p-1}$ is connected with
the fundamental solution of the nonlinear operator $\Delta_p$ which, when
$p\not= n$, is given by a multiple of the singular function
$|x-y|^{-\frac{n-p}{p-1}}$. Theorem \ref{T:LM} was inspired by a recent paper
of Crandall and Zhang \cite{CZ} in which the authors establish a superposition
principle stating that, for instance, in the range $2<p<n$ sums like 
\[
\sum \frac{A_j}{|x-a_j|^{\frac{n-p}{p-1}}},\ \ \ \ \ A_j\ge 0,
\]
are $p$-superharmonic in $\R^n$.

In this paper, we prove a superposition principle akin to Theorem \ref{T:LM},
but for a class of Riesz potentials which are naturally associated with a
given sub-Laplacian on a Lie group of Heisenberg type. We show that, quite
surprisingly, such Riesz potentials are intertwined with a nonlinear operator
which generalizes the $p$-Laplacian. Given the intricate geometry associated
with the complex structure of such groups, it is quite remarkable that, with a
substantial amount of additional technical complication, the Euclidean proof
in \cite{LM} extends virtually unchanged to this significantly wider class of 
ambients. This fact is a testament to the remarkable structure of such Lie
groups, and their extensive symmetry. 
  
To state our main result we recall that a Carnot group of step two is a
connected, simply connected Lie group $\G$ whose Lie algebra admits a
decomposition $\g = V_1\oplus V_2$ such that $[V_1,V_1] = V_2$ and $[V_1,V_2]
= \{0\}$. We assume that $\g$ has been endowed with an inner product
$\langle\cdot,\cdot\rangle$ with respect to which the two layers $V_1$ and
$V_2$ are orthogonal. Let $m$ be the dimension of $V_1$ and $k$ be the
dimension of $V_2$, and denote with $\{e_1,\ldots,e_m\}$ and
$\{\ve_1,\ldots,\ve_k\}$ orthonormal basis of $V_1$ and $V_2$ with respect to
$\langle\cdot,\cdot\rangle$.

The left-translation operator  $L_g:\G\to \G$ is defined by $L_g(g') = g
g'$. Its differential will be indicated by $dL_g$. Using the above orthonormal
basis we introduce left-invariant vector fields on $\G$ by letting 
\[
X_i(g) = dL_g(e_i),\ \ \  \ \ i=1,\ldots,m,
\]
\[
T_s(g) = dL_g(\ve_s),\ \ \ \ \ s=1,\ldots,k.
\]
Hereafter, we
agree that $\G$ is endowed with a left-invariant Riemannian metric with
respect  to which the vector fields $X_1,\ldots,X_m,T_1,\ldots,T_k$ constitute
an orthonormal basis.

The horizontal Laplacian (also known as sub-Laplacian) associated with the basis $\{e_1,\ldots,e_m\}$ is given by
\begin{equation}\label{L}
\mathcal L = \sum_{i=1}^m X_i^2.
\end{equation}
Observe that \eqref{L} fails to be elliptic at every point and that it is quite different from the (Riemannian) Laplace-Beltrami operator on $\G$, which is instead given by
\[
\Delta = \sum_{i=1}^m X_i^2 + \sum_{s=1}^k T_s^2.
\]
However, thanks to the grading assumption $[V_1,V_1] = V_2$, the vector fields $X_1,\ldots,X_m$, together with their commutators, generate the whole tangent bundle of $\G$, and therefore by H\"ormander's theorem \cite{H} the second order partial differential operator $\mathcal L$ is hypoelliptic.

Since the exponential mapping $\exp : \g \to \G$ is a surjective
diffeomorphism, we can define a global system of coordinates on $\G$ as
follows. Consider the analytic mappings $z:\G\to V_1$, $t:\G\to V_2$ defined
through the equation $g = \exp(z(g)+t(g))$. For each $i=1,\ldots,m$ we set
\[
z_i = z_i(g) = \langle z(g),e_i\rangle,
\]
whereas for $s=1,\ldots,k$ we let
\[
t_s = t_s(g) = \langle t(g),\ve_s\rangle.
\]
Henceforth, we will routinely omit the reference to the point $g\in \G$, and simply identify $g\in \G$ with its logarithmic coordinates $\exp^{-1}(g) = z(g) + t(g)$, where $z(g) = z_1(g)e_1+...+z_m(g) e_m$, $t(g) = t_1(g)\ve_1+...+t_k(g)\ve_k(g)$. To simplify notation, we will also routinely identity $z\in V_1$ with the vector $(z_1,...,z_m)\in \R^m$, and $t\in V_2$ with $(t_1,...,t_k)$, and we will write for $g\in \G$,
\[
g = (z,t) = (z_1,\ldots,z_m,t_1,\ldots,t_k).
\]

The number 
\begin{equation}\label{Q}
Q = m + 2k,
\end{equation}
associated with the nonisotropic dilations $\Delta_\lambda(z + t) = \lambda z
+ \lambda^2 t$ on $\g$, is called the homogeneous dimension of the group
$\G$. It plays the role of a dimension in the analysis of $\G$. In fact, if we
denote by $d\xi$ the standard $(m+k)$-dimensional Lebesgue measure on $\g$,
then one has $d(\Delta_\lambda(\xi)) = \lambda^Q d\xi$. Since a bi-invariant
Haar measure on $\G$ is obtained by pushing forward $d\xi$ via the exponential
mapping, such change of variable formula continues to be valid on $\G$ for the
natural nonisotropic dilations defined by the formula
\[
\delta_\lambda(g) = \exp\circ \Delta_\lambda\circ \exp^{-1} (g).
\]
We will indicate with $dg$ the Haar measure on $\G$. We thus have
\[
d(\delta_\lambda(g)) = \lambda^Q dg.
\]

We now consider the map
\[
J:V_2 \to \text{End}(V_1),
\]
which is uniquely specified by the identity 
\[
\langle J(t)z,z'\rangle = \langle [z,z'],t \rangle,\ \ \ \ z, z'\in V_1, t\in
V_2.
\]

\begin{definition}\label{D:Htype}
A Carnot group of step two $\G$ is called of \emph{Heisenberg type}, or of $H$-type, if for every 
$t\in V_2$ such that $|t| = 1$, the mapping $J$ is an orthogonal transformation of $V_1$ onto itself.
\end{definition}
From Definition \ref{D:Htype} we immediately see that, if $\G$ is of $H$-type, then
\begin{equation}\label{ht}
|J(t)z| = |z| |t|,\ \ \ \ \ \ z\in V_1, t\in V_2.
\end{equation}
From \eqref{ht}, and polarization, we obtain the important identity
\begin{equation}\label{ht2}
\langle J(t)z,J(t')z \rangle = \langle t,t' \rangle |z|^2,\ \ \ \  \ z\in V_1,
t, t'\in V_2.
\end{equation}

Groups of Heisenberg type were introduced by A. Kaplan in \cite{K} in
connection with questions of hypoellipticity of sub-Laplacians. They
constitute a natural generalization of the Heisenberg group $\Hn$, which (up
to group isomorphisms) one obtains from Definition \ref{D:Htype} when the
center $V_2$ is one-dimensional. For more information on the Heisenberg group
$\Hn$ we refer the reader to \cite{S2}. As it turns out, there is in nature a
plentiful supply of $H$-type groups since, for instance, they arise in the
Iwasawa decomposition of simple groups of rank one, see \cite{CDKR}.
  
The Folland--Kaplan gauge on $\G$ is defined by
\begin{equation}\label{N}
N(g) = (|z|^4 + 16 |t|^2)^{\frac{1}{4}}.
\end{equation}
We note explicitly that $N(g) = N(g^{-1})$. Although such function can be
defined in every Carnot group of step two, it has a special significance if the
group is of $H$-type. In such case, a remarkable discovery of Folland (for the
Heisenberg group $\Hn$) \cite{F1}, and Kaplan (for any group of Heisenberg
type) \cite{K}, shows that the fundamental solution of the sub-Laplacian
\eqref{L} is given by
\begin{equation}\label{gamma}
\Gamma(g,g') = - c(\G) N(g^{-1} g')^{2-Q},
\end{equation}
where $c(\G)>0$ is a suitable constant.

We now introduce the \emph{Riesz potential of order $\alpha$} on $\G$. Given $\rho\in C_0(\G)$, we set:
\begin{equation}\label{R}
R_\alpha(\rho)(g) = \int_{\G} \frac{\rho(g')}{N(g^{-1} g')^{Q-\alpha}} dg',\ \ \ \ 0<\alpha<Q.
\end{equation}
In view of \eqref{gamma}, we see that, up to a negative constant, $R_2(\rho)$
coincides with $\Gamma \star \rho$, where we have indicated with $\star$ the
group convolution in $\G$. But then, the potential identity $\mathcal
L(\Gamma \star \rho) =  \rho$, established in \cite{F2} for any Carnot group,
allows to conclude that  
\[
- \mathcal L(R_2(\rho)) = \rho.
\]
Hence, in a group of Heisenberg type $\G$ the Riesz potential $R_2$ plays the
same role of the classical Newtonian potential $I_2$ in $\R^n$.

These considerations led us to the question whether, in such groups, the
``linear'' objects $R_\alpha$ could  in any way be intertwined with a natural
nonlinear operator on $\G$ which has received considerable attention over the
past decade. Such important operator presents itself in the Euler-Lagrange
equation of the energy functional
\[
E_p(\phi) = \frac{1}{p} \int_\G |X\phi|^p dg,\ \ \ \ \ 1<p<\infty,
\]
in the Folland--Stein Sobolev embedding, see \cite{F2}. Here, we have indicated
with $X\phi = \sum_{i=1}^m X_i \phi X_i$ the subgradient of a function $\phi$,
and by $|X\phi| = (\sum_{i=1}^m (X_i \phi)^2)^{1/2}$ its length.  An elementary calculation shows that the
first variation of the energy $E_p$ leads  to the \emph{horizontal
  $p$-Laplacian}, which is defined on $\G$ via its action on a smooth function
$\phi$ by
\begin{equation}\label{Lp}
\mathcal L_{p} \phi = \sum_{i=1}^m X_i(|X \phi|^{p-2} X_i \phi),\ \ \ \ \ \ 1<p<\infty.
\end{equation}
Super-(sub-)solutions of the operator $\mathcal L_p$ must be defined in the
weak sense, see \cite{CDG1}, but at least for the range $p\ge 2$ of interest
in this paper, when the function $\phi$ is sufficiently smooth, then one can
see that $\phi$ is $p$-superharmonic ($p$-subharmonic) if $\mathcal L_p \phi
\le 0$ ($\mathcal L_p \phi\ge 0$).  

The answer to the above question is contained in the following theorem, which
is the main result of the present paper. 

\begin{theorem}\label{T:main}
Let $\G$ be a group of Heisenberg type with homogeneous dimension $Q$ given by \eqref{Q}. For a given $\rho\in C_0(\G)$, $\rho \ge 0$, the following three cases hold:
\begin{itemize}
\item[1)] If $2<p<Q$, then $R_{Q-\alpha}(\rho)$ is $p$-superharmonic when 
\[
0<\alpha \le \frac{Q-p}{p-1}.
\]
\item[2)] If $p>Q$, then $R_{Q-\alpha}(\rho)$ is $p$-subharmonic when 
\[
-\alpha \ge \frac{p-Q}{p-1}.
\]
If $p=\infty$, one may take $-\alpha \ge 1.$
\item[3)] If $p = Q$, then the function
\[
R_Q(\rho)(g) = \int_{\G} \rho(g') \log N(g^{-1} g') dg'.
\]
is $Q$-subharmonic.
\end{itemize}
\end{theorem}

The reader should notice the striking resemblance between Theorem \ref{T:main}
and its Euclidean predecessor Theorem \ref{T:LM}. Again, a notable aspect here
is the fact that the threshold exponent $\frac{Q-p}{p-1}$ is decided by that
of the singular solutions of $\mathcal L_p$, see Proposition \ref{P:pfs} in
Section \ref{S:s} below.

\

\paragraph{\bf Acknowledgements} Research for this paper was completed while
the authors were visitors in the Intensive Research Period ``Euclidean
Harmonic Analysis, Nilpotent Lie groups and PDE'' at the Centro di Ricerca
Matematica Ennio De Giorgi in April 2010. We are grateful to the De Giorgi
Center for the warm hospitality and the excellent working atmosphere.

\section{Calculus on $H$-type groups}\label{S:calc}

Let $\G$ be a Carnot group of step two with its left-invariant Riemannian
tensor with respect to which the vector fields $X_1,\ldots,X_m$,
$T_1,\ldots,T_k$ are orthonormal at every point. Given a smooth function
$u:\G\to\R$, the subgradient (or horizontal gradient) of $u$ is given by
$$
Xu = \sum_{i=1}^m X_i u X_i.
$$
Using the Baker-Campbell-Hausdorff formula, which for a Carnot group of step two becomes
\[
\exp(z +t) \exp(z' + t') = \exp(z + z' + t + t' + \frac{1}{2}[z,z']),
\]
it is easy to see that in the logarithmic coordinates $(z,t)$ one has
\begin{equation}\label{firstd}
X_i = \partial_{z_i} - \frac{1}{2} \sum_{s = 1}^k \sum_{j = 1}^m b^s_{ij} z_{j} \partial_{t_s},
\end{equation}
where for $i, j=1,\ldots,m, s=1,\ldots,k,$ we have indicated with
\begin{equation}\label{gc}
b^s_{ij} = \langle [e_i,e_j],\ve_s \rangle = \langle J(\ve_s)e_i,e_j\rangle,
\end{equation}
 the group constants.  Notice that for each $s=1,...,k$, the $m\times m$ matrix $\left[b_{ij}^s\right]$ is skew-symmetric, i.e., $b^s_{ii} = 0$, and that $b^s_{ji} = - b^s_{ij}$.
We also introduce another convenient notation
 \begin{equation}\label{B}
 B_{is} = B_{is}(g)  = - \sum_{j=1}^m b^s_{ij} z_j =   \langle
 J(\ve_s)z,e_i\rangle.
 \end{equation}
From \eqref{firstd} it is not difficult to obtain the commutator formula
\begin{equation}\label{commutator}
[X_i,X_{j}]  =  \sum_{s = 1}^k b^s_{ij} \partial_{t_s},
\end{equation}
which shows, in particular, that in contrast with the Euclidean case the \emph{horizontal Hessian matrix}
$$
X^2u = \left( X_i X_j u \right)_{i,j=1,\ldots,m},
$$
is not symmetric. More frequently, we will consider the {\it symmetrized horizontal Hessian}
$$
(X^2u)^*= \left( u_{,ij}\right)_{i,j=1,\ldots,m}, 
$$
where
\[
u_{,ij} = \frac{1}{2} (X_iX_j u + X_j X_i u) = X_i X_{j} - \frac{1}{2}
[X_i,X_j].
\]

In the proof of Theorem \ref{T:main} we will need to compute the horizontal
gradient and symmetrized horizontal Hessian for powers of the homogeneous norm
$N$ in \eqref{N}. In order to do this, we will build up to $N$ through several
preliminary steps. We introduce three more functions $\psi,\chi,a:\G\to\R$ by
the formulas
$$
\psi(g) := |z(g)|^2 = \sum_{i=1}^m z_i^2,
$$
$$
\chi(g) := |t(g)|^2 = \sum_{s=1}^k t_s^2,
$$
and
$$
a(g) := \psi(g)^2 + 16 \chi(g) = N(g)^4.
$$
We will compute first the horizontal derivatives
of $\psi$, then $\chi$, then $a$ and finally $N$.

\begin{lemma} 
Let $\G$ be a Carnot group of step two. (i) For any $i$ and $j$, $X_i(z_j) =
\delta_{ij}$. Then $X_i\psi=2z_i$ and $|X\psi|^2=4\psi$.
(ii) For any $i$, $j$ and $\ell$, $X_i X_j(z_\ell) =
0$. Then $X_i X_j \psi = 2\delta_{ij}$ and
$\psi_{,ij}=2\delta_{ij}$.
\end{lemma}

This lemma is completely trivial; the proof will be omitted. Using \eqref{firstd}, \eqref{gc} and \eqref{B} it is not difficult to prove the following result.

\begin{lemma}
Let $\G$ be a Carnot group of step two. (i) For any $i$ and $s$, we have $X_i
t_s = \frac12 B_{i s}$. Then
$$
X_i \chi = \langle J(t) z, e_i \rangle
$$
and thus $|X\chi|^2 = |J(t)z)|^2$. In particular, if $\G$ is of $H$-type, we
have from \eqref{ht} $|X\chi|^2 = \psi \chi$.

(ii) For any $i, j=1,\ldots,m$ and $s=1,\ldots,k$, one has \[
X_i X_j t_s = \tfrac12
\langle J(\ve_s) e_i, e_j \rangle = \frac{1}{2} b^s_{ij}.
\]
Then
$$
X_i X_j \chi = \langle J(t)e_i,e_j\rangle + \tfrac12
\sum_s B_{i s} B_{j s},
$$
and
$$
\chi_{,ij} = \tfrac12 \sum_s B_{i s} B_{j s}.
$$
\end{lemma}

This lemma is proved in \cite{DGN} (see Proposition 6.5). We observe in
passing some useful identities related to the coefficients $B_{i s}$, which
follow from the symmetry properties of the symplectic map $J$. First, for any
$s=1,\ldots,k$, we have 
\begin{equation}\label{b-useful-1}
\sum_i B_{i s} z_i = \langle J(\ve_s)z,z \rangle = 0.
\end{equation}
Second, when $\G$ is of $H$-type, then using \eqref{ht2} for any $r, s=1,\ldots,k$, we obtain
\begin{equation}\label{b-useful-2}
\sum_i B_{i r} B_{i s} = \langle J(\ve_r)z,J(\ve_s)z
\rangle = \langle \ve_r,\ve_s \rangle |z|^2 = \delta_{rs} \psi.
\end{equation}

The expression $|z|^2 z + 4 J(t) z$ will appear repeatedly in what
follows. We introduce a special notation for this expression, writing
\begin{equation}\label{A}
A = \psi z + 4 J(t) z.
\end{equation}
Note that 
\[
\langle A, z \rangle = \psi |z|^2 = \psi^2.
\]
When $\G$ is of $H$-type, then  using \eqref{ht} we find
\begin{equation}\label{A2}
|A|^2 = \psi^2 |z|^2 + 16 |z|^2 |t|^2 = \psi a.
\end{equation}

\begin{lemma}\label{L:a}
Let $\G$ be a Carnot group of step two, then
\[
Xa = 4A. 
\]
Moreover, for any $i, j=1,\ldots,m$,
$$
a_{,ij} = 4\psi \delta_{ij} + 8 ( z_i z_j + \sum_s B_{i s} B_{j s}).
$$
In particular, when $\G$ is of $H$-type, then 
\[
|Xa|^2 = 16 \psi a.
\]
\end{lemma}

Both parts of the lemma are direct computations using the formulas for the
derivatives of $\psi$ and $\chi$ in the preceding lemmas, together with the
formula for $a$.

Finally, we arrive at the formulas for the horizontal gradient and symmetrized
horizontal Hessian of the homogeneous norm $N$.

\begin{lemma}\label{L:N-lemma}
Let $\G$ be a Carnot group of step  two, then
$$
X N = N^{-3}  A.
$$
Furthermore, for any $i, j=1,\ldots,m$,
$$
N_{,ij} = N^{-7} \left( a \psi \delta_{ij} + 2a \left(z_i z_j + \sum_s B_{i s} B_{j s} \right) - 3 \langle
  A,e_i \rangle \langle A,e_j \rangle \right).
$$
When $\G$ is of $H$-type, then
\[
|XN|^2 = N^{-2} \psi.
\]
\end{lemma}

The proof of the first part of Lemma \ref{L:N-lemma} is an immediate
consequence of Lemma \ref{L:a} and the identities $X_i N = \frac14 a^{-3/4}
X_i a$ and
$$
N_{,ij} = \frac14 a^{-3/4} a_{,ij} - \frac3{16} a^{-7/4}X_i a X_j a.
$$
The second part follows directly from \eqref{A2}.

\section{Singular $p$-harmonic functions}\label{S:s}

The basic objects of study in this paper are the horizontal Laplacian
\eqref{L} and its nonlinear generalization, the horizontal $p$-Laplacian
\eqref{Lp}. We also consider the {\it horizontal $\infty$-Laplacian}
$$
\cL_\infty u = \sum_{i,j} u_{,ij} X_i u X_j u =
\langle (X^2u)^* Xu,Xu \rangle,
$$
which arises formally as a limit of $\cL_p u$ as $p\to\infty$. 

These nonlinear operators are only well-defined on functions which are twice
horizontally differentiable (and even on such functions, the action of
$\mathcal L_p$ is the range $1<p<2$ is not completely justified). The correct
notion of $p$-superharmonic (or $p$-subharmonic) function must be introduced
in the standard weak sense, see \cite{CDG1}. In this paper, however, we apply
the operator $\mathcal L_p$ only to functions which are sufficiently smooth,
and only in the range $p>2$. Therefore, we will perform all our computations
using strong derivatives. In this framework, we say that $u$ is {\it
  $p$-subharmonic} if $\cL_p u \ge 0$. Similarly, we say that $u$ is {\it
  $p$-superharmonic} if $\cL_p u \le 0$. We also have the following formula
\begin{equation}\label{LLL}
\cL_p u = |Xu|^{p-2} \cL u + (p-2) |Xu|^{p-4} \cL_\infty u
\end{equation}
which relates the horizontal $p$-Laplacian to the horizontal Laplacian and $\infty$-Laplacian.

As we have mentioned in Section \ref{S:intro}, one of the most remarkable
aspects of $H$-type groups is the fact that the fundamental solution of the
horizontal Laplacian is given by the formula \eqref{gamma}. Even more
interestingly, this phenomenon continues to hold in the nonlinear case. One
has in fact the following result from \cite{CDG2} (see also \cite{HH} for the
case $p=Q$).

\begin{proposition}\label{P:pfs}
For every $1<p<\infty$ the
function
\begin{equation}\label{pfs1}
\Gamma_p(g,g') = \Gamma_p(g',g) =
\begin{cases}
-\frac{p-1}{Q-p} \sigma_p^{-\frac{1}{p-1}}
N(g^{-1}g')^{\frac{p-Q}{p-1}},\ \ p\not= Q, \\
\\
-\ \sigma_p^{-\frac{1}{p-1}} \log\ N(g^{-1} g'),\ \ \ p = Q,
\end{cases}
\end{equation}
with $g'\not=g$, is a fundamental solution of \eqref{Lp} with
singularity at $g\in \G$.
\end{proposition}

In \eqref{pfs1} we have let $\sigma_p = Q \omega_p$, where
\[
\omega_p = \int_{B(e,1)} |X N(g)|^p dg.
\]
The meaning of Proposition \ref{P:pfs} is that for every $\phi\in C^\infty_0(\G)$ one has
\[
\phi(g) =  \int_\G |X\Gamma_p(g,g')|^{p-2} \langle X\Gamma_p(g,g'),X\phi(g') \rangle dg'.
\]

We note that $\Gamma(g,\cdot)\in C^\infty(\G\setminus\{g\})$, and furthermore
$$
\cL_p(\Gamma_p(g,\cdot)) \equiv 0 \qquad \mbox{in $\G \setminus \{g\}$.}
$$
Note that for every $g\not= g'$ we find
\[
\underset{p\to \infty}{\lim}\ \Gamma_p(g,g') = N( g^{-1} g' ).
\]
This suggests that we should also have
\begin{equation}\label{Linfinity}
\cL_\infty(N(g^{-1}\cdot)) \equiv 0 \qquad \mbox{in $\G\setminus \{g\}$.}
\end{equation}
The identity \eqref{Linfinity} is also true, see Proposition 6.4 in
\cite{DGN}. In fact, the homogeneous norm $N$ is a viscosity solution of the
horizontal $\infty$-Laplacian, see \cite{B}, \cite{BC} and \cite{W}.

\section{Proof of Theorem \ref{T:main}}\label{S:proofmain}

We are now prepared to give the proof of our main result.

\begin{proof}[Proof of Theorem \ref{T:main}]
Let $\alpha$ and $\rho$ be as in the statement of the theorem, and define
\begin{equation}\label{R1}
F(g) = R_{Q-\alpha}(\rho)(g) = \int_{\G} \frac{\rho(g')}{N((g')^{-1}
  g)^{\alpha}} dg' \qquad \mbox{if $0\not=\alpha<Q$}
\end{equation}
and
$$
F(g) = R_Q(\rho)(g) = \int_{\G} \rho(g') \log N(g^{-1} g') dg' \qquad \mbox{if
  $\alpha = 0$.}
$$
Here we used the fact that $N((g')^{-1} g) = N(g^{-1} g')$.

Let $2<p<\infty$; the case $p=\infty$ can be seen as a suitable limit of the
following computations. Using \eqref{LLL} we have
\begin{equation}\label{LL}
|XF|^{4-p} \cL_p F = |XF|^2 \cL F + (p-2) \cL_\infty F.
\end{equation}
If we now set $q = -\alpha$, then differentiating under the integral sign, and
using the left-invariance of the vector fields $X_1,\ldots,X_m$, we find that
\begin{equation}\label{first-deriv}
X_i F(g) = q \int_\G N((g')^{-1} g)^{q-1} X_i N((g')^{-1} g)
\rho(g') dg'
\end{equation}
and also that 
\begin{align}\label{second-deriv}
F_{,ij}(g)  &  = q \int_\G N((g')^{-1} g)^{q-1}
N_{,ij}((g')^{-1} g) \rho(g') dg'
\\ 
& + q(q-1) \int_\G N((g')^{-1} g)^{q-2} X_i N((g')^{-1}g) X_j
N((g')^{-1}g)\rho(g') dg'
\notag
\end{align}
if $q\ne 0$. (In the case $q=-\alpha=0$, the formulas must be modified by
removing the various factors of $q$ which appear in \eqref{first-deriv} and
\eqref{second-deriv}.)

Before proceeding any further, we introduce another simplifying notation. In
the remainder of the proof we will have need for all of the data computed in
the preceding lemmas, evaluated at different places in the group, e.g., at
$(g')^{-1}g$ for various points $g'$. In order to keep the formulas involved
from expanding out of control, we will write
$$
N_{g'}(g) := N((g')^{-1}g).
$$
After factoring out a constant multiple of $q^3$ from all terms, the right
hand side of \eqref{LL} is equal to the sum, over $i$ and $j$, of the
expression
\begin{equation*}\begin{split}
& \int_\G \left( (q-1) N_{g'}^{q-2} (X_i N_{g'})^2 + N_{g'}^{q-1}
  (N_{g'})_{,ii} \right) \rho(g') dg'
\biggl(\int_\G N_{g'}^{q-1} X_j N_{g'} \rho(g')dg'\biggr)^2 \\ 
& \quad + (p-2) \biggl( \int_\G \left( (q-1) N_{g'}^{q-2} X_i N_{g'}
  X_j N_{g'} + N_{g'}^{q-1} (N_{g'})_{,ij} \right) \rho(g') 
dg'\biggr) \\ & \qquad \biggl(\int_\G N_{g'}^{q-1} X_i N_{g'}
\rho(g') dg'\biggr) \biggl(\int_\G N_{g'}^{q-1} X_j N_{g'}
\rho(g') dg'\biggr). 
\end{split}\end{equation*}

Recall that there is an implicit variable $g$ in this expression; it is the
argument of the various terms involving the norm. Our claim is that this
expression, as a function of $g$, has a definite sign on all of $\G$,
depending only on the relative sizes of $p$ and $q = -\alpha$, as indicated
in the hypotheses.

We rewrite the preceding expression as a triple integral, introducing dummy
variables $g_1'$, $g_2'$ and $g_3'$. At this point we will simplify the
notation even further by abbreviating $N_1 := N_{g_1'}$, $\rho_1 :=
\rho(g_1')$ and so on. Then the right hand side of \eqref{LL} is equal to
\begin{equation}\label{LL1}
q^3 \sum_{i,j} \int\!\!\!\int\!\!\!\int \rho_1\rho_2\rho_3 N_1^{q-2}
N_2^{q-1} N_3^{q-1} \times \boldM_{ij} 
\end{equation}
where $\boldM_{ij}$ is equal to
\begin{equation*}\begin{split}
&(q-1) (X_i N_1)^2(X_j N_2)(X_j N_3) + N_1((N_1)_{,ii})(X_j N_2)(X_j N_3) \\
&\, + (p-2)(q-1) (X_i N_1)(X_j N_1)(X_i N_2)(X_j N_3) \\
&\quad + (p-2) N_1 ( (N_1)_{,ij} ) (X_i N_2)(X_j N_3).
\end{split}\end{equation*}
Here, the integral is taken with respect to the three variables
$g_1',g_2',g_3'$, over $\G\times\G\times\G$.

Our next task is to substitute the expressions from Lemma \ref{L:N-lemma} into
the above formula. The ensuing equation is a monster expression. In order to
gain some insight into the structure of the resulting formula, we separate it
into two parts which we treat independently. We write
$$
\boldM_{ij} = (q-1) \boldM^1_{ij} + \boldM^2_{ij},
$$
where
$$
\boldM^1_{ij} = (X_i N_1)^2(X_j N_2)(X_j N_3) +
(p-2)(X_i N_1)(X_j N_1)(X_i N_2)(X_j N_3)
$$
and
$$
\boldM^2_{ij} = N_1 ((N_1)_{,ii})(X_j N_2)(X_j N_3) +
(p-2) N_1 ( (N_1)_{,ij} ) (X_i N_2)(X_j N_3).
$$
Using Lemma \ref{L:N-lemma} we compute
\begin{equation*}\begin{split}
&\boldM^1_{ij} = N_1^{-6} N_2^{-3} N_3^{-3} \times \\ & \, \times \biggl[
\langle A_1,e_i \rangle^2 \langle A_2,e_j \rangle \langle A_3,e_j
\rangle + (p-2) \langle A_1,e_i \rangle \langle A_1,e_j \rangle
\langle A_2,e_i \rangle \langle A_3,e_j \rangle \biggr]
\end{split}\end{equation*}
and
\begin{equation*}\begin{split}
&\boldM^2_{ij} = N_1^{-6} N_2^{-3} N_3^{-3} \times \\ & \, \times \biggl[
\biggl( a_1\psi_1 + 2a_1(z_1)_i^2+2a_1\sum_s(B_1)_{i s}^2-3\langle
A_1,e_i\rangle^2 \biggr) \langle A_2,e_j\rangle \langle A_3,e_j
\rangle \biggr. \\ & \quad +(p-2) \biggl( a_1\psi_1\delta_{ij} +
2a_1(z_1)_i(z_1)_j+ \biggr. \\ & \quad \biggl. \biggl. +
2a_1\sum_s(B_1)_{i s}(B_1)_{j s} -3\langle A_1,e_i\rangle\langle
A_1,e_j\rangle \biggr) \langle A_2,e_i \rangle \langle A_3,e_j
\rangle \biggr].
\end{split}\end{equation*}
Here, we have continued to use subscripts to denote the result of evaluation of
the various functions $a$, $\psi$, $A$, etc.\ at the points $(g')^{-1}g$,
i.e., $a_1 := a((g_1')^{-1} g)$ and so on. However, one notation might unfortunately generate confusion. The reader should pay attention to the fact that henceforth $z_1$ means $z((g_1')^{-1} g)$, not the first component of the vector $z = (z_1,\ldots,z_m)$. Accordingly, the notation $(z_1)_i$ indicates the $i$-th component of the vector $z((g_1')^{-1} g)\in V_1$.

Returning to \eqref{LL1}, we perform the summations in $i$ and
$j$. Here, we make use of the fact that the vectors $e_i$ form an
orthonormal basis for $V_1$, and that the functions $A_1,A_2,A_3$ take values
in $V_1$. Consequently, we obtain that the right hand side of \eqref{LL} is
equal to the sum of the following two terms (which are obtained by summing the
expressions for $(q-1)\boldM^1_{ij}$ and $\boldM^2_{ij}$, respectively):
\begin{equation}\label{M12}\begin{split}
&q^3(q-1) \int\!\!\!\int\!\!\!\int \rho_1\rho_2\rho_3 N_1^{q-8} N_2^{q-4}
N_3^{q-4} \times \\
& \qquad \times \biggl( |A_1|^2 \langle A_2,A_3 \rangle + (p-2) \langle A_1,A_2
\rangle \langle A_1,A_3 \rangle \biggr)
\end{split}\end{equation}
and, recalling that $m=\dim V_1$,
\begin{equation}\begin{split}\label{M22}
&q^3 \int\!\!\!\int\!\!\!\int \rho_1\rho_2\rho_3 N_1^{q-8} N_2^{q-4} N_3^{q-4}
\times \\
& \, \times \biggl[ \biggl( (m+2) a_1\psi_1 +2a_1\sum_{i,s}(B_1)_{i s}^2-3|A_1|^2 \biggr) \langle A_2,A_3\rangle \biggr. \\
& \, + (p-2) \biggl( a_1\psi_1\langle A_2,A_3\rangle + 2a_1\langle
A_2,z_1\rangle \langle A_3,z_1\rangle \biggr. \\
& \, \biggl. \biggl. + 2a_1 \sum_{s}\langle
A_2,J(\ve_s) z_1\rangle \langle A_3,J(\ve_s) z_1 \rangle -3\langle
A_1,A_2\rangle\langle A_1,A_3\rangle \biggr) \biggr].
\end{split}\end{equation}
We also used the identity
\begin{equation}\label{BA}
\sum_{i,j}(B_1)_{i s}\langle A_2,e_i\rangle (B_1)_{j
  s} \langle A_3,e_j \rangle
=
\langle A_2,J(\ve_s) z_1\rangle \langle A_3,J(\ve_s) z_1 \rangle,
\end{equation}
valid for each $s=1,\ldots,k$. Identity \eqref{BA} can be easily verified
from \eqref{B} and from the orthonormality of $\{e_1,\ldots,e_m\}$. By
\eqref{b-useful-2}, $\sum_{i}(B_1)_{i s}^2 = \psi$ for each
$s=1,\ldots,k$. Hence $\sum_{i,s}(B_1)_{i s}^2 = k\psi$, where $k$ denotes the
dimension of the second layer $V_2$. 

Using \eqref{Q} and
\eqref{A2}, we regroup the various terms in \eqref{M22} to obtain
\begin{equation}\begin{split}\label{M24}
&q^3 \int\!\!\!\int\!\!\!\int \rho_1\rho_2\rho_3 N_1^{q-8} N_2^{q-4} N_3^{q-4}
\times \\
& \, \times \biggl[ (Q+p-3) |A_1|^2 \langle A_2,A_3\rangle + 2(p-2)a_1\langle
A_2,z_1\rangle \langle A_3,z_1\rangle \biggr. \\
& \, \biggl. + 2(p-2)a_1\sum_{s}\langle A_2,J(\ve_s)z_1\rangle \langle
A_3,J(\ve_s)z_1 \rangle -3(p-2)\langle A_1,A_2\rangle\langle
A_1,A_3\rangle \biggr].
\end{split}\end{equation}
We now recombine \eqref{M12} and \eqref{M24}, obtaining the following formula
for the right hand side of \eqref{LL}:
\begin{equation}\begin{split}\label{M3}
&q^3 \int\!\!\!\int\!\!\!\int \rho_1\rho_2\rho_3 N_1^{q-8} N_2^{q-4} N_3^{q-4} \times \biggl[ (Q+p+q-4) |A_1|^2 \langle A_2,A_3\rangle \biggr. \\
& \quad + 2(p-2)a_1\langle A_2,z_1\rangle \langle A_3,z_1\rangle \\
& \qquad \biggl. + 2(p-2)a_1\sum_{s}\langle A_2,J(\ve_s)z_1\rangle \langle A_3,J(\ve_s)z_1 \rangle \\ & \quad \qquad + (q-4)(p-2)\langle A_1,A_2\rangle\langle A_1,A_3\rangle \biggr].
\end{split}\end{equation}
Observe that the dependence on the $g_2'$ and $g_3'$ variables occurs only in the expressions $A_2$ and $A_3$, and that the integrand is bilinear in those expressions. It is thus natural to introduce the new function
$$
\mathcal K(g) = \int_\G \rho(g') N_{g'}^{q-4}(g) A_{g'}(g) \, dg'
$$
obtained by integrating $A_{g'}$ against the coefficient $\rho(g')
N_{g'}^{q-4}$ which appears in \eqref{M3}. Performing the integration in the
$g_2'$ and $g_3'$ variables simplifies \eqref{M3} as follows:
\begin{equation}\begin{split}\label{M4}
&q^3 \int \rho_1 N_1^{q-8} \biggl[ (Q+p+q-4) |A_1|^2 |\mathcal K|^2 + (q-2)(p-2)
\langle A_1,\mathcal K\rangle^2 \biggr. \\ 
& \, \biggl. + 2(p-2) \biggl( a_1 \langle \mathcal K,z_1\rangle^2 +
a_1\sum_{s} \langle \mathcal K,J(\ve_s)z_1\rangle^2 - \langle A_1,\mathcal K\rangle^2
\biggr) \biggr].
\end{split}\end{equation}
Observe that
$$
(Q+p+q-4) |A_1|^2 |\mathcal K|^2 + (q-2)(p-2) \langle A_1,\mathcal K\rangle^2 \ge 0
$$
for all $A_1$ and $\mathcal K$ by the Cauchy--Schwarz inequality, provided that
$$
q = - \alpha \ge \frac{p-Q}{p-1}
$$
as hypothesized in the statement of the theorem. To complete the proof, we
will show that the expression
\begin{equation}\label{estimate0}
a_1 \langle \mathcal K,z_1\rangle^2 + a_1\sum_{s} \langle
\mathcal K,J(\ve_s)z_1\rangle^2 - \langle A_1,\mathcal K\rangle^2
\end{equation}
is always nonnegative. Using \eqref{A}, we compute
$$
\langle A_1,\mathcal K\rangle = \psi_1 \langle \mathcal K,z_1 \rangle + 4 \langle
\mathcal K,J(t_1) z_1 \rangle,
$$
where, as above for $z_1$, the notation $t_1$ indicates $t((g_1')^{-1} g)$, and not the first component of the vector $t = (t_1,\ldots,t_k)$.
Since $a=\psi^2+16\chi$, the quantity in \eqref{estimate0} is equal to
\begin{equation*}\begin{split}
&16 \chi_1 \langle \mathcal K,z_1 \rangle^2 - 8 \psi_1 \langle \mathcal K,z_1
\rangle \langle \mathcal K,J(t_1)z_1 \rangle \\ & \qquad - 16 \langle
\mathcal K,J(t_1)z_1 \rangle^2 + ((\psi_1)^2 + 16 \chi_1) \sum_s \langle
\mathcal K,J(\ve_s)z_1 \rangle^2,
\end{split}\end{equation*}
which we write as the sum of the following two expressions:
\begin{equation*}\begin{split}
I := 16 \chi_1 \langle \mathcal K,z_1 \rangle^2 - 8 \psi_1 \langle \mathcal
K,z_1 \rangle \langle \mathcal K,J(t_1) z_1 \rangle + (\psi_1)^2
\sum_s \langle \mathcal K,J(\ve_s)z_1 \rangle^2
\end{split}\end{equation*}
and
$$
II := 16 \chi_1 \sum_s \langle \mathcal K,J(\ve_s)z_1 \rangle^2 - 16 \langle
\mathcal K,J(t_1)z_1 \rangle^2.
$$
The proof is now finished by observing that
$$
I = \sum_s \biggl( 4 (t_s)_1 \langle \mathcal K,z_1 \rangle - \psi_1 \langle
  \mathcal K,J(\ve_s)z_1 \rangle \biggr)^2,
$$
while $II$ is nonnegative by another application of the Cauchy--Schwarz
inequality. 
The assertion about $I$ follows by recognizing that
$$
\sum_s t_s \langle \mathcal K,J(\ve_s)z \rangle = \sum_s \langle t,\ve_s
\rangle \langle [z,\mathcal K],\ve_s \rangle = \langle t,[z,\mathcal K]\rangle
= \langle \mathcal K,J(t)z \rangle.
$$
In conclusion, we have shown that the integral in \eqref{M4} is nonnegative,
under the stated assumptions on $p$ and $q = - \alpha$. The factor of $q^3$
then ensures that the $p$-Laplacian takes on the appropriate sign (based on
whether $p<Q$ or $p>Q$). Note that $q$ is strictly negative in case 1) of the
theorem, while $q$ is strictly positive in case 2) of the theorem. In case 3)
($p=Q$) the argument is identical, except that the overall factor of $q$ in
\eqref{first-deriv} and \eqref{second-deriv} should be dropped. After that,
setting $q=0$ and repeating the argument yields the desired conclusion.
This completes the proof of the theorem.

\end{proof}

In closing, we mention that Theorem \ref{T:main} can be generalized to the case of Riesz potentials, 
\[
R_\alpha(\mu) = \int_\G \frac{d\mu(g')}{N(g^{-1} g')^{Q-\alpha}},
\]
of rather general Radon measures $\mu$.
If we consider a Radon measure $\mu$ on $\G$, then, similarly to what has been done in \cite{LM}, we can show that Theorem \ref{T:main} continues to be valid for the Riesz potential  $R_{Q-\alpha}(\mu)$, provided that $\mu$ fulfills the growth assumption
\[
\int_{N(g)\ge 1} \frac{d\mu(g)}{N(g)^{\alpha}} < \infty.
\]
We omit the details.


\end{document}